\documentclass[12pt,a4paper]{amsart}
\usepackage{amsmath,mathrsfs,amssymb,amsthm, latexsym, amsfonts,alltt}
\usepackage[T1]{fontenc}
\usepackage{xy}
\usepackage{xr}
\usepackage{enumerate}
\usepackage[latin1]{inputenc}
\usepackage[francais]{babel}
\include{xy}
\xyoption{all}
\numberwithin{equation}{section}

\nopagebreak

\theoremstyle{plain}
\newtheorem{theorem}{Th\'eor\`eme}[section]
\newtheorem{corollary}[theorem]{Corollaire}
\newtheorem{lemma}[theorem]{Lemme}
\newtheorem{proposition}[theorem]{Proposition}
\theoremstyle{definition}
\newtheorem{defi}[theorem]{D\'efinition}
\newtheorem{remark}[theorem]{Remarque}

\newtheorem{question}[theorem]{Question} 
\DeclareMathOperator{\supp}{supp}

\DeclareMathOperator{\Irr}{Irr}

\DeclareMathOperator{\JH}{JH}
\DeclareMathOperator{\Hom}{Hom}

\DeclareMathOperator{\gr}{gr}
\DeclareMathOperator{\Jac}{Jac}
\DeclareMathOperator{\ind}{ind}

\DeclareMathOperator{\ry}{Nrd}
\DeclareMathOperator{\sy}{s.s.}
\author[Alberto M\'inguez]{Alberto M\'inguez}\thanks{Partially supported by MTM2004-07203-C02-01 and FEDER}\address{Alberto M\'inguez, Laboratoire de Math\'ematiques, Universit\'e Paris-Sud, B\^at 425
91405 Orsay Cedex, France, CNRS UMR 8628. \\ URL: {\rm  http://www.math.u-psud.fr/$\sim$minguez/}} 
\email{minguez@clipper.ens.fr}

\begin{document}
%\setcounter{tocdepth}{2} \setcounter{section}{-1} 
%\begin{titlepage}
\title{Sur l'irr\'eductibilit\'e d'une induite parabolique}

%\author{Alberto M\'{i}nguez\footnote{Partially supported by MTM2004-07203-C02-01 and FEDER}}
%\maketitle
\begin{abstract}
Let $F$ be a non-Archimedean locally compact field and let $D$ be a central division algebra over $F$. Let $\pi_1$ and $\pi_2$ be respectively two smooth irreducible representations of ${\rm GL}(n_1,D)$ and ${\rm GL}(n_2,F)$, $n_1, n_2 \geq 0$. In this article, we give some sufficient conditions on $\pi_1$ and $\pi_2$ so that the parabolically induced representation of $\pi_1 \otimes \pi_2$ to ${\rm GL}(n_1+n_2,D)$ has a unique irreducible quotient. In the case where $\pi_1$ is a cuspidal representation, we compute the Zelevinsky's parameters of such a quotient in terms of parameters of $\pi_2$. This is the key point for making explicit Howe correspondence for dual pairs of type II (\textit{cf.} \cite{Mi1}).
\end{abstract}
\maketitle
%\end{titlepage}
{\bf Codes MSN:} 22E50, 22E35.
\vspace{1cm}

\section*{Introduction}
Dans l'\'etude des repr\'esentations irr\'eductibles d'un groupe r\'eductif sur un corps local non archim\'edien $F$, on est amen\'e \`a consid\'erer aussi des repr\'esentations r\'eductibles. En g\'en\'eral, elles ne sont pas semi-simples et il est tr\`es int\'eressant de construire de telles repr\'esentations ayant un unique quotient (ou sous-repr\'esentation) irr\'eductible. C'est dans cette lign\'ee, par exemple, que l'on trouve les classifications de Langlands et de Zelevinsky du type: \textit{"l'unique quotient irr\'eductible de..."} ou \textit{"l'unique sous-repr\'esentation irr\'eductible de..."}.

Cet article est motiv\'e par la question suivante: \textit{ 
Soit $D$ une alg\`{e}bre \`{a} division de centre $F$ de dimension finie sur $F$. Si $\pi_1$ et $\pi_2$ sont deux repr\'esentations lisses irr\'eductibles de ${\rm GL}(n_1,D)$ et ${\rm GL}(n_2,F)$ respectivement, est-ce que l'induite parabolique, not\'ee $\pi_1 \times \pi_2$, de $\pi_1 \otimes \pi_2$ \`a ${\rm GL}(n_1+n_2,D)$ poss\`ede un unique quotient irr\'eductible?}

Nous utilisons des techniques des foncteurs de Jacquet, avec des consid\'erations de combinatoire, pour donner des conditions suffisantes pour que la r\'eponse soit positive. 

Donnons quelques premi\`eres applications de ce r\'esultat g\'en\'eral: la repr\'esentation $\pi_1 \times \pi_2$ admet, en particulier, un unique quotient irr\'eductible si:
\begin{enumerate}
\item La repr\'esentation $\pi_1$ est un caract\`ere de ${\rm GL}(n_1,D)$. C'est une question qui appara\^it naturellement dans la correspondance de Howe. Ceci prouve une ancienne conjecture de M.-F. Vign\'eras, \textit{cf.} \cite[Conjecture 3.III.6]{MVW}.
\item La repr\'esentation $\pi_1$ est une repr\'esentation cuspidale. Ceci permet, dans \cite{Mi1}, de rendre explicite la correspondance de Howe pour les paires duales de type II.
\item Les repr\'esentations $\pi_1$ et $\pi_2$ sont essentiellement de carr\'e int\'egrable. Ceci permet de donner (\textit{cf.} \cite{Mi2}) une preuve simple, compl\`etement combinatoire, de la classification de Zelevinsky \cite{Z1}, des repr\'esentations irr\'eductibles de ${\rm GL}(n,D)$, en termes des segments.
\end{enumerate}
Remarquons, pour comprendre l'importance de la question qui motive cet article, que parmi d'autres applications, l'unicit\'e du quotient irr\'eductible impliquerait, en particulier, la conjecture (U0) de Tadic (qui vient d'\^etre prouv\'ee dans \cite{Se4}) qui permet de d\'eterminer le dual unitaire de ${\rm GL}(n,D)$.

Donnons maintenant plus de d\'etails concernant les diff\'erentes sections de l'article:

Dans la section \ref{1}, on introduit les notations. Dans la	 section \ref{lemageo}, on utilise le lemme g\'eom\'etrique combinatoire de \cite{Z1}, pour donner une condition suffisante pour que l'induite parabolique du produit tensoriel de plusieurs repr\'esentations irr\'eductibles n'ait qu'un seul quotient irr\'eductible. En section \ref{bloques}, on introduit les foncteurs de Jacquet et, en section \ref{sectadic}, on rappelle la classification de Tadic des repr\'esentations irr\'eductibles de ${\rm GL}(n,D)$.

Dans la section \ref{irr}, on combine ces r\'esultats pour donner le th\'eor\`eme principal de cet article (Th\'eor\`eme \ref{l}). Dans la section \ref{3}, on calcule, avec la classification de la section \ref{sectadic}, les param\`etres de l'unique quotient irr\'eductible de $\pi_1 \times \pi_2$, quand la repr\'esentation $\pi_1$ est cuspidale, en termes des param\`etres de $\pi_2$.

On utilise ce calcul, dans la section \ref{aap}, pour tirer quelques cons\'equences: on donne une condition combinatoire n\'ecessaire et suffisante d'irr\'eductibilit\'e de $\pi_1 \times \pi_2$, quand $\pi_1$ est cuspidale (Th\'eor\`eme \ref{irredu}). De plus, il nous permet de montrer (Th\'eor\`eme \ref{geometrique}) une conjecture g\'eom\'etrique de Tadic \cite[Conjecture 3.6]{Tadic} sur l'involution de Zelevinsky ainsi que la g\'en\'eralisation de la description combinatoire de cette involution au cas des formes int\'erieures de ${\rm GL}(n,F)$, ce qui facilite les calculs de \cite{Tad2}.

Dans l'appendice, on montre (Corollaire \ref{comb}) un lemme jouant un r\^ole cl\'e dans le calcul de \cite{Mi1} de la correspondance th\^eta explicite dans le cas des paires duales de type II. Il serait tr\`es int\'eressant de trouver une preuve plus simple, sans devoir utiliser tous les calculs des sections pr\'ec\'edentes, pour, peut-\^etre, g\'en\'eraliser les r\'esultats de \cite{Mi1} aux paires duales de type I. Remarquons que dans le cas des repr\'e\-sen\-ta\-tions des groupes orthogonaux, la repr\'esentation induite parabolique \`a partir du produit tensoriel de deux repr\'esentations cuspidales peut avoir deux sous-modules irr\'eductibles (\textit{cf.} \cite{Moe}).

Je voudrais particuli\`erement remercier Guy Henniart et Colette M\oe glin pour leurs nombreux conseils et id\'ees. Je remercie  aussi Florent Benaych-Georges, Goran Muic, Hiroshi Saito et Vincent S\'echerre pour les remarques et corrections int\'eressantes qu'ils m'ont faites \`a propos de cet article.

%Dans la section \ref{1}, on introduit les notations de Zelevinsky pour des repr\'esentations de ${\rm GL}_n(D)$, $D$ \'etant une alg\`ebre \`a division, ce qui nous permettra d'\'etendre les r\'esultats de \cite{MW} pour des repr\'esentations d'une alg\`ebre centrale simple. Ces r\'esultats d\'eviennent beaucoup plus simples \`a montrer avec l'article sur l'involution de Zelevinsky \'ecrit par Aubert \cite{Aubert}.

%Dans la section \ref{2}, on donne des conditions suffisantes pour que l'induite de deux repr\'esentations irr\'eductibles ait un seul sous-module irr\'eductible (corollaires \ref{irred} et \ref{irred2}).
%On conjecture que l'induite de deux repr\'esentations irr\'eductibles $\tau$ et $\pi$ de ${\rm GL}_n(D)$ et ${\rm GL}_m(D)$ a un unique sous-module irr\'eductible et que la multiplicit\'e de ce sous-module dans $\JH(\tau \times \pi) $ est \'egale \`a $1$. Cette conjecture g\'en\'eralise la conjecture (U0) de Tadic \cite{Tadic}.

%Finalement, dans la section \ref{3}, on pr\'ecise le lemme II.9. de \cite{MW}, on calcule les param\`etres de Zelevinsky de l'unique sous-module irr\'eductible de $\tau \times \pi$, $\tau$ \'etant une repr\'esentation cuspidale, en fonction de ceux de $ \pi $, ce qui nous permet de trouver une caract\'erisation combinatoire de la situation o\`u l'induite d'une repr\'esentation irr\'eductible et une cuspidale est encore irr\'eductible. 

\section{Pr\'eliminaires}\label{1}
Soient $F$ un corps commutatif localement compact non archim\'{e}dien de caract\'{e}ristique
r\'{e}siduelle $p>0$, $D$ une alg\`{e}bre \`{a} division de centre $F$ de dimension finie $d^2$ sur $F$.

On note $\mathcal{M}_{n}$ l'ensemble de matrices $n \times n$ \`a coefficients dans $D$ et $\ry :\mathcal{M}_n\rightarrow F$ la norme r\'{e}duite. Le groupe ${\rm GL}_n(D)$ des matrices inversibles dans $\mathcal{M}_{n}$ sera not\'e $G_n$. Le groupe trivial sera not\'e $G_0$. On note $\nu=|\ry|_F$, la valeur absolue de la norme r\'eduite (par abus de notation on ne fera pas distinction entre $\nu$ agissant sur $G_n$ pour diff\'erent $n$).

A toute partition (ordonn\'ee) $\alpha=\left( n_1,\dots,n_r\right)$, $n_i \geq 0$, de l'entier $n$, correspond une d\'ecomposition en blocs des matrices carr\'ees d'ordre $n$. On notera $G_{\alpha}$ le sous-groupe de $G_n$ form\'e des matrices inversibles diagonales par blocs, $P_{\alpha}$ (resp. $\overline{P}_{\alpha}$) le sous-groupe form\'e des matrices triangulaires sup\'erieures (resp. inf\'erieures) par blocs, et $U_{\alpha}$ le sous-groupe de $P_{\alpha}$ form\'e des \'el\'ements dont les blocs diagonaux sont des matrices unit\'e. Le sous-groupe $\overline{P}_{\alpha}$ est conjugu\'e \`a $P_{\overline{\alpha}}$ dans $G_n$ avec $\overline{\alpha}=\left( n_r,\dots,n_1\right)$. 

Dans cet article on ne consid\'erera que des repr\'esentations lisses complexes et le mot \textit{repr\'esentation} voudra toujours dire \textit{repr\'esentation lisse complexe}. On notera $\Irr (G_n)$ l'ensemble des classes d'\'equivalence des repr\'esentations irr\'eductibles de $G_n$. On note $\Irr$ l'union disjointe
\begin{equation*}
\Irr=\bigcup_{n\geq0}\Irr (G_n),
\end{equation*}
et $\mathcal{C}$ le sous-ensemble de $\Irr$ form\'e de repr\'esentations cuspidales. Le support cuspidal de $\pi \in \Irr$ sera not\'e $\supp (\pi)$.

On note $r_{(n_1,\dots,n_r),n}$ (resp. $\overline{r}_{(n_1,\dots,n_r),n}$) le foncteur de Jacquet normalis\'e associ\'e au parabolique standard $P_{\alpha}$ (resp. $\overline{P}_{\alpha}$). 

Soient $\rho_i \in \Irr(G_{n_i})$, $1\leq i \leq r$. On note $\rho_1\times \dots \times \rho_r$ la repr\'esentation $\ind^{G_n}_{P_{\alpha}}\left(\rho_1, \otimes \dots \otimes \rho_r \otimes 1_{U_{\alpha}}\right) $ induite parabolique normalis\'ee.

Soit $\pi$ une repr\'esentation de $G_n$; on a un isomorphisme canonique (r\'eciprocit\'e de Frobenius): \begin{equation}\label{frob}
\Hom_{G_n} \left( \pi, \rho_1\times \dots \times \rho_r \right) \simeq \Hom_{G_\alpha} \left( r_{(n_1,\dots,n_r),n}(\pi), \rho_1 \otimes \dots \otimes \rho_r \right).
\end{equation}

On dispose aussi d'un isomorphisme de r\'eciprocit\'e \textit{\`a la Casselman} (\textit{cf.} \cite[Theorem 20]{ber}):
\begin{equation}\label{frobcas}
\Hom_{G_n} \left( \rho_1\times \dots \times \rho_r , \pi\right) \simeq \Hom_{G_\alpha} \left( \rho_1 \otimes \dots \otimes \rho_r ,\overline{r}_{(n_1,\dots,n_r),n}(\pi) \right).
\end{equation}

On notera $\mathcal{R}_n$ le groupe de Grothendieck de la cat\'egorie des $G_n$-modules de longueur finie identifi\'e au $\mathbb{Z}$-module libre qui a pour base $\left( \pi \right)_{\pi \in \Irr(G_n)}$. Le sous-semigroupe de $\mathcal{R}_n$ qui consiste en des sommes finies $\pi_1 + \dots + \pi_k$ o\`u $\pi_i \in \Irr(G_n)$, $k \geq 0$ sera not\'e $\mathcal{R}^+_n$. Posons 
\begin{eqnarray*}
\mathcal{R}&=&\bigoplus_{n \geq 0} \mathcal{R}_n ,\\
\mathcal{R}^+&=&\bigoplus_{n \geq 0} \mathcal{R}^+_n.
\end{eqnarray*}
Si $\pi_1, \pi_2 \in \mathcal{R}$ on note $\pi_1 \leq \pi_2$ si $\pi_2 - \pi_1 \in \mathcal{R}^+$. 

On note $\sy (\pi)$ (ou $\JH(\pi)$) l'image de $\pi$ dans $\mathcal{R}$ pour toute repr\'esentation de longueur finie $\pi$.

Si $\pi$ est une repr\'esentation de $G_n$ on notera $\gr (\pi) =n$ et $ \widetilde{\pi}$ la contragr\'ediente de $\pi$.

Si $\pi$ est une repr\'esentation de longueur finie, alors $$\sy (\pi) = \bigoplus_{\tau_i \in \Irr, 1\leq i \leq r} m_i \tau_i, \qquad \text{o\`u les $\tau_i$ sont distincts.}$$ On dit que les $\tau_i$ forment une suite de composition de $\pi$ et que  $m_i$ est la multiplicit\'e de $\tau_i$ dans $\pi$.
\section{Lemme g\'eom\'etrique combinatoire}\label{lemageo}
On rappelle ici les r\'esultats de \cite[1.6.]{Z1} et on d\'eduit quelques premiers lemmes simples qui seront utilis\'es dans la suite.

Soient $\beta, \gamma$ deux partitions de $n$, $\beta = \left( n_1,\dots,n_r\right)$, $\gamma =\left( m_1,\dots,m_s\right)$ et pour $i \in \{1, \dots r\}$ soit $\rho_i$ une repr\'esentation de $G_{n_i}$. On veut calculer une suite de composition de $r_{\left( m_1,\dots,m_s\right),n} \left( \rho_1\times \dots \times \rho_r \right)$. 

Notons $M^{\beta, \gamma}$ l'ensemble des matrices $b=(b_{i,j})$ telles que
\begin{enumerate}
\item Les $b_{i,j}$ sont des entiers non n\'egatifs,
\item\label{2} $\sum_jb_{i,j}=n_i$ pour tout $i=1,\dots,r$; $\sum_ib_{i,j}=m_j$ pour tout $i=1,\dots,s$.
\end{enumerate}

Fixons $b \in M^{\beta, \gamma}$ et notons $\beta_i$ la partition $\left( b_{i1},\dots,b_{is}\right)$ de $n_i$ et $ \gamma_j$ la partition $\left( b_{1j},\dots,b_{rj}\right)$ de $m_j$. 

Les $r_{\beta_i,n_i}(\rho_i)$ sont de longueur finie. Supposons alors $\JH\left(r_{\beta_i,n_i}(\rho_i)\right) = \left\{\sigma^{(1)}_{i}, \sigma^{(2)}_{i}, \dots \sigma^{(r_i)}_{i} \right\}$ o\`u $$\sigma^{(k)}_{i}=\sigma^{(k)}_{i1}\otimes \dots \otimes \sigma^{(k)}_{is}, \text{ }\sigma^{(k)}_{ij}\in \Irr\left( G_{b_{ij}}\right), \qquad  k=1,\dots,r_i.$$
Pour tous $k_1, \dots, k_r$ on pose
$$\sigma_{j}=\sigma^{(k_1)}_{1j}\times \sigma^{(k_2)}_{2j} \times \dots \times \sigma^{(k_r)}_{rj}, \qquad 1 \leq k_j \leq r_j,$$
repr\'esentation de $G_{m_j}$, et 
$$\sigma(k_1, \dots, k_r)=\sigma_{1}\otimes \sigma_{2}\otimes \dots \otimes \sigma_{s},$$
repr\'esentation de $G_{\gamma}$.

Alors, d'apr\`es \cite[1.6.]{Z1}, les $\sigma(k_1, \dots, k_r)$ quand on fait varier les $(k_1, \dots, k_r)$ et $b \in M^{\beta, \gamma}$ forment une suite de composition de la repr\'esentation $r_{\left( m_1,\dots,m_s\right),n} \big( \rho_{1} \times \dots \times \rho_r \big)$.

Une premi\`ere cons\'equence imm\'ediate est la proposition ci-dessous:
\begin{proposition}\label{cons} 
Avec les notations pr\'ec\'edentes, supposons que pour tout $i$, $1 \leq i \leq r$, toute partition $\beta_i=\left( b_{i1}, b_{i2}\right)$ de $n_i$ et tout $\sigma^{(k)}_{i}=\sigma^{(k)}_{i1}\otimes \sigma^{(k)}_{i2} \in \JH\left(r_{\beta_i,n_i}(\rho_i)\right) $ on ait
\begin{enumerate}
\item Ou bien $\supp\left(\sigma^{(k)}_{i1} \right) \nsubseteq \underset{1 \leq j \leq i-1}{\bigcup} \supp \left( \rho_j \right)$;
\item Ou bien $\supp\left(\sigma^{(k)}_{i2} \right) \nsubseteq \underset{i+1 \leq j \leq r}{\bigcup} \supp \left( \rho_j \right)$.
\end{enumerate}
%Soient $n,t$ deux entiers positifs, $\rho \in \Irr(G_t)$ et $\pi \in \Irr(G_n)$ deux repr\'esentations telles que pour tout $i \leq n$ et $\tau_1 \otimes \tau_2 \in \JH \left( r_{(i,n-i),n} \left( \pi \right) \right)$, $\tau_1 \in \Irr (G_i )$ et $\tau_2 \in \Irr (G_{n-i} )$, 
%$\supp \left( \tau_1 \right) \nsubseteq \supp \left( \rho \right)$. 
Alors la repr\'esentation $\rho_1\otimes \dots \otimes \rho_r $ appara\^it avec multiplicit\'e $1$ dans $\JH \left( r_{\beta, n} \left(\rho_1\times \dots \times \rho_r \right) \right)$.
\end{proposition}
\begin{proof}
D'apr\`es ce qui pr\'ec\`ede,  les $\sigma(k_1, \dots, k_r)$, quand on fait varier $(k_1, \dots, k_r)$ et $b \in M^{\beta, \beta}$, forment une suite de composition de $r_{\beta, n} \left(\rho_1\times \dots \times \rho_r \right) $.

Si $b$ est la matrice diagonale dans $M^{\beta, \beta}$ d'\'el\'ements diagonaux $n_1, $ $n_2, \dots, n_r$, alors $\sigma(k_1, \dots, k_r)=\rho_1\otimes \dots \otimes \rho_r $ et $k_i=1$ pour tout $i$.

Il suffit donc de montrer que pour tout $b \in M^{\beta, \beta}$, $b$ non diagonale, et tout $(k_1, \dots, k_r)$, $\rho_1\otimes \dots \otimes \rho_r $ n'est pas un sous-quotient de $\sigma(k_1, \dots, k_r)$. Or, si $b$ n'est pas diagonale (par la condition \eqref{2} de la page \pageref{2} sur les matrices $b$), il existe $j <i$ et $j'>i'$ tels que $b_{ij}$ et $b_{i'j'}$ soient non nuls. Soient
\begin{eqnarray*}
&&i_0= \min \left\{ i: \exists j<i, b_{ij} \neq 0\right\} \qquad j_0= \min \left\{ j: b_{i_0j} \neq 0\right\},\\
&&i'_0= \max \left\{ i': \exists j'>i', b_{i'j'} \neq 0\right\} \qquad j'_0= \max \left\{ j': b_{i'_0j'} \neq 0\right\}.
\end{eqnarray*}

Par l'hypoth\`ese (1), puisque $j_0<i_0$, pour tout $k_{i_0}$, $\supp \left( \sigma^{(k_{i_0})}_{i_0j_0} \right) \nsubseteq \supp \left( \rho_{j_0}\right)$ et donc, $\rho_{j_0}$ ne peut pas \^etre un sous-quotient de $\sigma_{j_0}$. Ainsi, $\rho_1\otimes \dots \otimes \rho_r $ n'est pas un sous-quotient de $\sigma(k_1, \dots, k_r)$. 

Par l'hypoth\`ese (2), puisque $j'_0>i'_0$, pour tout $k_{i'_0}$, $\supp\left( \sigma^{(k_{i'_0})}_{i'_0j'_0} \right) \nsubseteq \supp \left( \rho_{j'_0}\right) $ et donc, $\rho_{j'_0}$ ne peut pas \^etre un sous-quotient de $\sigma_{j'_0}$. Ainsi, $\rho_1\otimes \dots \otimes \rho_r $ n'est pas un sous-quotient de $\sigma(k_1, \dots, k_r)$. 
%
%On utilise les notations pr\'ec\'edentes
%\newline
%les partitions $\beta=\gamma=(t,n)$,
%\newline et
%$b_i=
%\left(
%\begin{array}{cc}
% i&t-i \\
% t-i &n-t+i 
%\end{array}
%\right)
%$
%avec $\max\left\{0,t-n\right\} \leq i \leq t $. 
%
%Fixons $i$ avec $\max\left\{0,t-n\right\} \leq i < t $ . Alors \newline
%$\beta_1=(i,t-i)$ est une partition de $t$ \newline
%$\beta_2=(t-i,n-t+i)$ est une partition de $n$ \newline
%$\lambda_1=(i,t-i)$ partition de $t$ \newline
%$\lambda_1=(t-i,n-t+i)$ partition de $n$ \newline
%$\JH \left( r_{\beta_1} \left( \rho \right) \right)= \left\{ \sigma_1^{k} \otimes \sigma_2^{k} \right\}_k$\newline
%$\JH \left( r_{\beta_2} \left( \rho \right) \right)= \left\{ \tau_1^{j} \otimes \tau_2^{j} \right\}_j$
%
%$ r_{t,n} \left( \rho \times \pi \right) $ est compos\'e de repr\'esentations de la forme 
%$$( \sigma_1^{k} \times \tau_1^{j}) \otimes (\sigma_2^{k}\times \tau_2^{j})$$
%Or, d'apr\`es nos hypoth\`eses, $\supp( \sigma_1^{k} \times \tau_1^{j}) \nsubseteq \supp (\rho)$ et donc aucun sous-quotient de $( \sigma_1^{k} \times \tau_1^{j}) \otimes (\sigma_2^{k}\times \tau_2^{j})$ ni, \`a fortiori, de $F_{i}\left( \rho \otimes \pi \right) $ n'est isomorphe \`a $\rho \otimes \pi$, pour $i <t$.
%Mais $F_{t} \left( \rho \otimes \pi \right) =\rho \otimes \pi$, d'o\`u le r\'esultat.
\end{proof}

\begin{corollary}\label{mmm1}
Avec les notations pr\'ec\'edentes, supposons que pour tout $i$, $1 \leq i \leq r$, toute partition $\beta_i=\left( b_{i1}, b_{i2}\right)$ de $n_i$ et tout $\sigma^{(k)}_{i}=\sigma^{(k)}_{i1}\otimes \sigma^{(k)}_{i2} \in \JH\left(r_{\beta_i,n_i}(\rho_i)\right) $, on ait
\begin{enumerate}
\item Ou bien $\supp\left(\sigma^{(k)}_{i1} \right) \nsubseteq \underset{1 \leq j \leq i-1}{\bigcup} \supp \left( \rho_j \right)$;
\item Ou bien $\supp\left(\sigma^{(k)}_{i2} \right) \nsubseteq \underset{i+1 \leq j \leq r}{\bigcup} \supp \left( \rho_j \right)$.
\end{enumerate}
%Soient $n,t$ deux entiers positifs, $\rho \in \Irr(G_t)$ et $\pi \in \Irr(G_n)$ deux repr\'esentations telles que pour tout $i \leq n$ et $\tau_1 \otimes \tau_2 \in \JH \left( r_{(i,n-i),n} \left( \pi \right) \right)$, $\tau_1 \in \Irr (G_i )$ et $\tau_2 \in \Irr (G_{n-i} )$, 
%$\supp \left( \tau_1 \right) \nsubseteq \supp \left( \rho \right)$. 
Alors $\rho_1\times \dots \times \rho_r $ a un seul sous-module irr\'eductible et sa multiplicit\'e dans l'induite est \'egale \`a $1$.
\end{corollary}
\begin{proof}
C'est une cons\'equence du lemme qui suit et de la proposition pr\'ec\'edente.
\end{proof}
\begin{lemma}\label{mult1}
Supposons que $\rho_1\otimes \dots \otimes \rho_r $ appara\^it avec multiplicit\'e $1$ dans 
\begin{eqnarray*}
&&\JH \left( r_{\beta, n} \left(\rho_1\times \dots \times \rho_r \right) \right) \\
&\text{(resp. }& \JH \left( \overline{r}_{\beta, n} \left(\rho_1\times \dots \times \rho_r \right) \right).\text{)}
\end{eqnarray*}
Alors $\rho_1\times \dots \times \rho_r $ a un seul sous-module irr\'eductible et sa multiplicit\'e dans l'induite est \'egale \`a $1$.
\end{lemma}
\begin{proof}
Supposons qu'il existe $\pi_1$ et $\pi_2$ deux sous-modules irr\'eductibles de $\rho_1\times \dots \times \rho_r $ et posons $\pi= \pi_1 \oplus \pi_2$. Ainsi
$$\dim \left(\Hom_{G_n} \left( \pi, \rho_1\times \dots \times \rho_r\right) \right)\geq 2.$$
Par \eqref{frob}, on trouve 
$$\dim \left(\Hom_{G_\beta} \left( r_{\beta,n} \left(\pi\right), \rho_1\otimes \dots \otimes \rho_r\right) \right)\geq 2.$$
Or, par l'exactitude du foncteur de Jacquet, 
$$\JH \left( r_{\beta, n} \left( \pi \right) \right) \leq \JH \left( r_{\beta, n} \left(\rho_1\times \dots \times \rho_r \right) \right).$$
On trouve donc que $\rho_1\otimes \dots \otimes \rho_r $ appara\^it avec multiplicit\'e au moins $2$ dans $\JH \left( r_{\beta, n} \left(\rho_1\times \dots \times \rho_r \right) \right)$ ce qui est absurde, par hypoth\`ese.
\end{proof}

\section{Les blocs dans $Fin(G_{n})$}\label{bloques}
Pour un ensemble $X$, $M(X)$ sera l'ensemble de tous les multi-ensem\-bles dans $X$, \textit{i.e,} des fonctions $m:X \rightarrow \mathbb{N}$ \`a support fini. On d\'efinit la somme de multi-ensembles de fa\c{c}on naturelle.

Si $\Omega \in M\left( \mathcal{C} \right)$, on pose $$\left| \Omega \right|= \underset{\rho \in \mathcal{C}}{\sum} \Omega \left(\rho \right)\gr \left( \rho\right).$$

Soit $\Omega$ un multi-ensemble de repr\'esentations cuspidales, avec $\left| \Omega \right|=n$. On note $Fin(G_{n})$ la sous-cat\'egorie pleine de $Alg \left(G_n \right)$ des repr\'esentations de longueur finie et $Fin \left(\Omega \right)$ la sous-cat\'egorie pleine de $Fin(G_{n})$ telle que pour tout sous-quotient irr\'eductible $\pi$ de tout \'el\'ement dans $ Fin \left(\Omega \right)$, $\supp (\pi)=\Omega$. 

Alors (\textit{cf.} \cite[Theorem 7.3.2]{cas}) la cat\'egorie $Fin(G_{n})$ est produit direct des \textit{blocs} $Fin \left(\Omega \right)$ quand $\Omega$ parcourt tous les multi-ensembles de repr\'esentations cuspidales de $G_n$ avec $\left| \Omega \right|=n$.

On veut dire par cela, 
\begin{enumerate}
\item Tout $\pi \in Fin \left(G_n \right)$ est une somme directe de repr\'esentations $\pi_i \in Fin \left(\Omega_i \right)$, o\`u $\Omega_i \neq \Omega_j$ si $i \neq j$.
\item $\Hom_{G_n} \left( \pi_i, \pi_j\right)=0$ si $\pi_i \in Fin \left(\Omega_i \right)$ et $\Omega_i \neq \Omega_j$.
\end{enumerate}

%Ce r\'esultat est sans doute vrai dans le cas banal pour $G$ quelconque (non seulement ${\rm GL}_n(F)$) mais nous n'avons trouv\'e des r\'ef\'erences dans la litt\'erature.

Soient $n,t$ deux entiers positifs, $t<n$ et soit $\pi$ une repr\'esentation de longueur finie de $G_n$. La repr\'esentation $r_{\left(t,n-t\right),n}(\pi)$, \'etant de longueur finie, se d\'ecompose alors:
$$r_{\left(t,n-t\right),n}(\pi)=\overset{r}{ \underset{i=1}{\bigoplus}}\overset{s}{ \underset{j=1}{\bigoplus}}\Pi_{i,j}$$
o\`u, pour tous $1 \leq i \leq r$ et $1 \leq j \leq s$, $\Pi_{i,j}$ est une repr\'esentation de longueur finie de $G_t \times G_{n-t}$ telle que tout sous-quotient irr\'educitble de $\Pi_{i,j}$ est la forme $\pi_i \otimes \pi'_j$ avec $\pi_i \in \Omega_i$ et $\pi'_j \in \Omega'_j$ o\`u $\left\{ \Omega_i\right\}_{1 \leq i \leq r}$ et $\left\{ \Omega'_i\right\}_{1 \leq j \leq s}$ sont deux sous-ensembles de $M\left( \mathcal{C} \right)$ tels que $\left| \Omega_i \right|=t$, $\left| \Omega'_j \right|=n-t$, pour tous $1 \leq i \leq r$ et $1 \leq j \leq s$, et $\Omega_i \neq \Omega_{i'}$ si $i \neq i'$, et $\Omega'_j \neq \Omega'_{j'}$ si $j \neq j'$.

%Soient $n,m$ deux entiers positifs, $n<m$, $\pi$ est une repr\'esentation de longueur finie de $G_m$.
%Il existe $\Omega_i, \Omega'_j \in M\left( \mathcal{C} \right)$, $\left| \Omega_i \right|=n$, $\left| \Omega'_j \right|=m-n$, $1 \leq i \leq r$, $1 \leq j \leq s$, $\Omega_i \neq \Omega_{i'}$ si $i \neq i'$ et $\Omega'_j \neq \Omega'_{j'}$ si $j \neq j'$ tels que la repr\'esentation $r_{\left(n,m-n\right),m}(\pi)$ de $G_n \times G_{m-n}$ (\'etant de longueur finie) se d\'ecompose:
%$$r_{\left(n,m-n\right),m}(\pi)=\overset{r}{ \underset{i=1}{\bigoplus}}\overset{s}{ \underset{j=1}{\bigoplus}}\Pi_{i,j}$$
%avec tout sous-quotient irr\'educitble de $\Pi_{i,j}$ de la forme $\pi_i \otimes \pi'_j$ o\`u $\pi_i \in \Omega_i$ et $\pi'_j \in \Omega'_j$.

Si $\rho$ une repr\'esentation irr\'eductible de $G_t$, avec $\supp \left( \rho \right) = \Omega_i$, on pose  $$\Jac_{\rho}(\pi)=\underset{1 \leq j \leq s}{\bigoplus}\Pi_{i,j},$$ et si $\rho'$ une repr\'esentation irr\'eductible de $G_{n-t}$ avec $\supp \left( \rho' \right) = \Omega'_j$, on pose $$\overline{\Jac}_{\rho'}(\pi)=\underset{1 \leq i \leq r}{\bigoplus}\Pi_{i,j}.$$

Remarquons que $\Jac_{\rho}(\pi)$ et $\overline{\Jac}_{\rho'}(\pi)$ ne d\'ependent que du support cuspidal de $\rho$ et $\rho'$, respectivement. Remarquons aussi que on obtient le foncteur $\overline{\Jac} _{\rho}$ \`a partir du foncteur $\Jac_{\rho}$ en changeant dans la d\'efinition de $\Jac_{\rho}$ le foncteur de Jacquet par son foncteur oppos\'e (d'o\`u la notation).

\begin{proposition}\label{JA}
Supposons qu'il existe $\rho \in \Irr(G_{t})$, $V \in \Irr(G_{n-t})$ telles que $\rho \otimes V$ soit un sous-quotient de $r_{\left(t,n-t\right),n}(\pi)$. Il existe alors $\rho'$ une repr\'esentation irr\'eductible avec $\supp \left( \rho' \right)= \supp \left( \rho \right) $ et une repr\'esentation $V' \in \Irr(G_{n-t})$ telles que $\pi$ soit un sous-module de $\rho' \times V'$. 
\end{proposition}
\begin{proof}
Par hypoth\`ese, $\Jac_{\rho}(\pi)\neq 0$. Soit $\rho' \otimes V'$ un quotient irr\'eductible de $\Jac_{\rho}(\pi)$. Alors $\rho'$ est une repr\'esentation irr\'eductible avec $\supp \left( \rho' \right)= \supp \left( \rho \right) $ et, la compos\'ee du morphisme non nul dans 
$$\Hom_{G_{(t,n-t)}} \left( \Jac_{\rho}(\pi), \rho' \otimes V' \right) \neq 0$$
avec le morphisme surjectif canonique de $r_{\left(t,n-t\right),n}(\pi)$ vers $\Jac_{\rho}(\pi)$ montre que
$$\Hom_{G_{(t,n-t)}} \left( r_{\left(t,n-t\right),n}(\pi), \rho' \otimes V' \right) \neq 0,$$
et, par r\'eciprocit\'e de Frobenius on trouve le r\'esultat cherch\'e.
\end{proof}
\section{Repr\'esentations de ${\rm GL}_n(D)$}\label{sectadic} 
On va rappeler ici quelques r\'esultats de \cite{Tadic} et \cite{Mi2}. Dans \cite[\textsection 2]{Tadic} (voir aussi \cite[Theorem 3.4]{Se4} pour une preuve diff\'erente valable en toute caract\'eristique), on montre qu'il existe une fonction qui, \`a chaque repr\'esentation cuspidale $\rho$ associe un entier strictement positif $s_{\rho}$ tel que, si $\rho_1$ et $\rho_2$ sont deux repr\'esentations cuspidales de ${\rm GL}_{n_1}(D)$ et ${\rm GL}_{n_2}(D)$, $n_1, n_2 \geq 1$ alors $\rho_1\times \rho_2$ est r\'eductible si, et seulement si 
$$\rho_1 = \nu^{s_{\rho_1}} \rho_2\text{ ou } \rho_1 = \nu^{-s_{\rho_1}} \rho_2. $$

%\begin{enumerate}
%\item[R1.] Si $\rho_1$ et $\rho_2$ sont deux repr\'esentations cuspidales de ${\rm GL}_i(D)$ et ${\rm GL}_j(D)$, $i+j=n$ et $\rho_1\times \rho_2$ est r\'eductible, alors $i=j$ et il existe un unique caract\`ere (\`a inversion pr\`es) $\chi_{\rho_1}$ tel que
%\begin{equation*}
%\rho_1=\chi_{\rho_1} \rho_2\text{ ou } \rho_1=\chi_{\rho_1}^{-1} \rho_2.
%\end{equation*}
%\item[R2.] Pour toute repr\'esentation cuspidale $\rho$, il existe un unique entier positif ${s_{\rho}}$ tel que 
%\begin{equation*}
%\chi_{\rho}=\nu^{s_{\rho}}.
%\end{equation*}
%\item[R3.] Si $\rho_1$ et $\rho_2$ sont deux repr\'esentations cuspidales de ${\rm GL}_i(D)$ et ${\rm GL}_j(D)$, et $\rho_1\times \rho_2$ est irr\'eductible, alors 
%\begin{equation*}
%\rho_1\neq \nu^{s_{\rho_1}} \rho_2\text{ et } \rho_1\neq \nu^{-s_{\rho_1}} \rho_2. 
%\end{equation*}
%\end{enumerate}
%Elles sont montr\'ees dans \cite[\textsection 2]{Tadic}. 
Pour toute repr\'esentation $\rho$ cuspidale on pose $ \nu_{\rho}= \nu^{s_{\rho}}.$

Soit $\rho\in\mathcal{C}$, $n\in \mathbb{N}^{\ast}$. On pose 
$$\Delta=\left\{\rho, \nu_{\rho}\rho,\dots, \nu^{n-1}_{\rho}\rho\right\}.$$ On appelle $\Delta$ un segment et l'ensemble de tous les segments sera not\'e $S$. 

On dit que $\Delta=\left\{\rho, \nu_{\rho}\rho,\dots , \nu^{n-1}_{\rho}\rho\right\}$, $\Delta'=\left\{\rho', \nu_{\rho'}\rho',\dots,\nu^{n'-1}_{\rho'}\rho'\right\}$ sont li\'es si $\Delta \cup \Delta'$ est encore un segment et $\Delta \nsubseteq \Delta'$ et $\Delta' \nsubseteq \Delta$.

On dit que $\Delta$ pr\'ec\`ede $\Delta'$ s'ils sont li\'es et il existe $\tau \in \Delta$ tel que $\rho' =\tau \nu_{\tau}$. On dit que $\Delta \geq \Delta'$ s'il existe $r \in \mathbb{N}^\ast$ tel que 
$$\rho = \nu_\rho ^r \rho', \text{ ou bien }\rho=\rho' \text{ et }n \geq n'.$$ 
L'ordre ainsi d\'efini n'est total que si l'on se restreint \`a des segments inclus dans l'ensemble $\left\{\rho \nu_{\rho}^{t} : t\in \mathbb{Z} \right\}$, o\`u $\rho$ est fix\'e..

Un \textit{multisegment} est un multi-ensemble de segments de la forme ci-dessus. On dira qu'un multisegment est rang\'e si l'on peut l'identifier \`a l'ensemble indic\'e $(\Delta_1, \dots, \Delta_r)$ o\`u
$$\Delta_r \nless \Delta_{r-1} \nless \dots \nless \Delta_2 \nless \Delta_1.$$

\begin{proposition}
A chaque segment $\Delta=\left\{\rho, \nu_{\rho}\rho,\dots,\nu^{n-1}_{\rho}\rho\right\}$, $\rho$ \'etant une repr\'esentation cuspidale de $G_p$, on peut associer des repr\'esentations irr\'eductibles $ \left<\Delta \right>$ et $ \left<\Delta \right>^t$ telles que
\begin{enumerate}
\item $ \left<\Delta \right>$ (resp. $ \left<\Delta \right>^t$) est l'unique sous-module (resp. quotient) irr\'eductible de $\rho \times \nu_{\rho}\rho \times \dots \times \nu^{n-1}_{\rho}\rho$.
\item $r_{(p,\dots,p),np}\left( \left<\Delta \right> \right)= \rho \otimes \nu_{\rho}\rho \otimes \dots \otimes \nu^{n-1}_{\rho}\rho$ (resp. $r_{(p,\dots,p),np}\left( \left<\Delta \right>^t \right) = \nu^{n-1}_{\rho}\rho \otimes \dots \otimes \nu_{\rho}\rho \otimes \rho$).
\end{enumerate}
Les repr\'esentations $ \left<\Delta \right>$ et $ \left<\Delta \right>^t$ sont aussi caract\'eris\'ees par la propri\'et\'e (2). L'ensemble de repr\'esentations de la forme $ \left<\Delta \right>^t$ est l'ensemble des repr\'esentations essentiellement de carr\'e int\'egrable. 
\end{proposition}
\begin{proof}
\textit{Cf.} \cite[2.7.]{Tadic}
\end{proof}

A partir d'un argument d'unitarit\'e on montre le th\'eor\`eme suivant, cl\'e dans toute la construction qui suit:
\begin{theorem}\label{mover}
Les conditions suivantes sont \'equivalentes
\begin{enumerate}
\item Pour tous $1 \leq i,j \leq r$, les segments $\Delta_i$ et $\Delta_j$ ne sont pas li\'es.
\item $\left< \Delta_1 \right> \times \dots \times \left< \Delta_n \right>$ est irr\'eductible.
\item $\left< \Delta_1 \right>^t \times \dots \times \left< \Delta_n \right>^t$ est irr\'eductible.
\end{enumerate}
\end{theorem}
\begin{proof}
1 \'equivaut \`a 3 par \cite[2.5.]{Tadic}. 2 \'equivaut \`a 3 par \cite[Corollaire 3.9.(b)]{Aubert}, car les repr\'esentations $\left< \Delta \right>$ et $\left< \Delta \right>^t $ se correspondent par l'involution de Zelevinsky.
\end{proof}

Les th\'eor\`emes ci-dessous, classifient toutes les repr\'esentations irr\'eductibles de $G_n$ en fonction des repr\'esentations cuspidales de $G_i$, $i \leq n$.
\begin{theorem}\label{Z}
\begin{enumerate}
\item Soient $\Delta_1,\dots,\Delta_r$ des segments et supposons que, si $i <j$, $\Delta_i$ ne pr\'ec\`ede pas $\Delta_j$. Alors la repr\'esentation $ \left<\Delta_1 \right>^t\times \dots \times \left<\Delta_r \right>^t$ admet un unique quotient irr\'eductible. Il est not\'e $ \left<\Delta_1,\dots ,\Delta_r \right>^t$. La multiplicit\'e de $ \left<\Delta_1,\dots ,\Delta_r \right>^t$ dans $ \JH \left(\left<\Delta_1 \right>^t\times \dots \times \left<\Delta_r \right>^t\right)$ est \'egale \`a $1$.
\item Les repr\'esentations  $ \left<\Delta_1,\dots ,\Delta_r \right>^t$ et $ \left<\Delta'_1,\dots ,\Delta'_{r'} \right>^t$ sont \'equivalentes si, et seulement si, les suites $( \Delta_1 ,  \dots  ,$ $\Delta_r)$ et $(\Delta'_1,\dots,\Delta'_{r'})$ sont \'egales \`a l'ordre pr\`es.
\item Toute repr\'esentation irr\'eductible de $G_n$ peut s'\'ecrire sous la forme $ \left<\Delta_1,\dots ,\Delta_r \right>^t$.
\end{enumerate}
\end{theorem}
\begin{proof}
\textit{Cf.} \cite[\textsection 2]{Tadic}
\end{proof}
C'est une param\'etrisation \textit{\`a la Langlands}.
De m\^eme, on trouve un th\'eor\`eme similaire, \textit{\`a la Zelevinsky}, quand on change le mot "quotient" par "sous-module":
\begin{theorem}
\begin{enumerate}
\item Soient $\Delta_1,\dots,\Delta_r$ des segments et supposons que, si $i <j$, $\Delta_i$ ne pr\'ec\`ede pas $\Delta_j$. Alors la repr\'esentation $ \left<\Delta_1 \right>\times \dots \times \left<\Delta_r \right>$  admet une unique sous-repr\'esentation  irr\'eductible. On la note $ \left<\Delta_1,\dots ,\Delta_r \right>$. La multiplicit\'e de $ \left<\Delta_1,\dots ,\Delta_r \right>$ dans $\JH \big( \left<\Delta_1 \right>\times \left<\Delta_2 \right>\times\dots \times \left<\Delta_r \right> \big)$ est \'egale \`a $1$.
\item Les repr\'esentations $ \left<\Delta_1,\dots ,\Delta_r \right>$ et $ \left<\Delta'_1,\dots ,\Delta'_{r'} \right>$ sont \'equivalentes si, et seulement si, les suites $( \Delta_1 ,  \dots  ,$ $\Delta_r)$ et $(\Delta'_1,\dots,\Delta'_{r'})$ sont \'egales \`a l'ordre pr\`es.
\item Toute repr\'esentation irr\'eductible de $G_n$ peut s'\'ecrire sous la forme $ \left<\Delta_1,\dots ,\Delta_r \right>$.
\end{enumerate}
\end{theorem}
\begin{proof}
\textit{Cf.} \cite[2.3.6]{Mi2}
\end{proof}
\begin{remark}\label{prac}
Si $\left\{ \Delta_1 ,\dots,\Delta_r \right\}$ est un multi-segment rang\'e, avec $\Delta_1 = \left\{ \rho, \nu_{\rho}\rho,\dots,\nu^{n-1}_{\rho}\rho\right\}$, alors 
$$\overline{\Jac}_{\rho}\left( \left<\Delta_1,\dots ,\Delta_r \right>^t\right) \neq \emptyset.$$
En effet, $ \left<\Delta_1,\dots ,\Delta_r \right>^t$ \'etant un quotient de $ \left<\Delta_1 \right>^t\times \dots \times \left<\Delta_r \right>^t$ , $ \left<\Delta_1 \right>^t\otimes \dots \otimes \left<\Delta_r \right>^t \in \overline{r}\left( \left<\Delta_1,\dots ,\Delta_r \right>^t \right)$, avec $\overline{r}$ le foncteur de Jacquet associ\'e au parabolique appropri\'e.
\end{remark}
\begin{defi}\label{zze}
L'automorphisme d'alg\`ebre de $\mathcal{R}$ qui envoie la repr\'esentation $ \left<\Delta_1,\dots ,\Delta_r \right>$ vers  $ \left<\Delta_1,\dots ,\Delta_r \right>^t$ est appel\'ee l'involution de Zelevinsky.
\end{defi}
\begin{remark}\label{rectas} Soit $X$ un sous-ensemble maximal de $\mathcal{C}$ v\'erifiant la propri\'et\'e suivante: si $\rho_1,\rho_2 \in X$, $n\in \mathbb{Z}$ et $\rho_1= \nu^{n} \rho_2$ , alors $n=0$. Notons $\mathcal{R}(\rho)$ le sous-ensemble de $\mathcal{R}$ form\'e des repr\'esentations dont le support est inclus dans l'ensemble $\left\{\rho \nu_{\rho}^{t} : t\in \mathbb{Z} \right\}$. Alors \cite[\textsection 3]{Tadic}
$$ \mathcal{R}=\bigotimes_{\rho \in X}\mathcal{R}(\rho).$$
\end{remark}

\section{Unicit\'e du sous-module}\label{irr}
%on suppose que les hypoth\`eses R1 et R3 sont v\'erifi\'ees pour $D$, $R$ et $N$, o\`u $N$ est la dimension du plus grand $G_i$ qui apparaisse dans le texte (rappelons que la d\'efinition de \textit{banal} d\'epend de ce $N$, fixer \`a priori $N$ compliquerait \'enorm\'ement les notations). 
Le but de cette section est de donner des conditions suffisantes sur deux repr\'esentations irr\'eductibles $\pi$ et $\rho$ pour que leur induite parabolique $\pi \times \rho$ n'ait qu'un seul sous-module irr\'eductible. C'est un probl\`eme tr\`es int\'eressant: si c'\'etait toujours le cas, \textit{i.e.} si l'induite parabolique du produit tensoriel de deux repr\'esentations irr\'eductibles avait toujours un seul sous-module irr\'eductible alors cela montrerait 
la conjecture (U0) de \cite{Tadic}, \textit{i.e.} le fait que l'induite $\pi \times \rho$ de deux repr\'esentations irr\'eductibles unitaires de ${\rm GL}_i(D)$ et ${\rm GL}_j(D)$ respectivement reste toujours irr\'eductible. Cette conjecture vient d'\^etre prouv\'ee par V. S\'echerre (\textit{cf.} \cite{Se4}).

%on suppose que les hypoth\`eses R1 et R3 sont v\'erifi\'ees pour $D$, $R$ et $N$, o\`u $N$ est la dimension du plus grand $G_i$ qui apparaisse dans le texte (rappelons que la d\'efinition de \textit{banal} d\'epend de ce $N$, fixer \`a priori $N$ compliquerait \'enorm\'ement les notations). 

\begin{theorem}\label{l}
Soient $\Delta_1, \dots , \Delta_r$ des segments \underline{non li\'es} v\'erifiant la condition suivante
\begin{equation}\label{C} \text{Si $i \neq j$, alors ou bien $\Delta_i = \Delta_j$ ou bien $\Delta_i \cap \Delta_j =\emptyset $}.
\end{equation}
Soit $\rho=\left< \Delta_1, \dots , \Delta_r \right>^t$ (resp. $\left< \Delta_1, \dots , \Delta_r \right>$) et $\pi \in \Irr(G_n)$. Alors $\pi \times \rho$ a un seul sous-module irr\'eductible et il appara\^it avec multiplicit\'e $1$ dans $\JH\left( \pi \times \rho\right)$.
\end{theorem}
\begin{defi}
Soient $\pi \in \Irr(G_n)$ et $\rho \in \Irr (G_p)$, $p\leq n$. On d\'efinit l'entier $l_{\pi}^{\supp(\rho)}$ par 
\begin{eqnarray*}
l_{\pi}^{\supp(\rho)}=\max&&\hspace{-.7cm}\big\{ i : \exists \tau_1 \in \Irr(G_{n-i}), \tau_2 \in \Irr(G_{i}) \\ && \hspace{-1cm}\text{ avec } \Hom_{G_n}(\pi, \tau_1 \times \tau_2) \neq 0, \text{et } \supp(\tau_2) \subset \supp (\rho) \big\}.
\end{eqnarray*}
\end{defi}
\begin{remark}
On a aussi 
\begin{eqnarray*}
l_{\pi}^{\supp(\rho)}=\max&&\hspace{-.7cm} \big\{ i : \exists \tau'_1 \in \Irr(G_{n-i}), \tau'_2 \in \Irr(G_{i}) \\ &&\hspace{-1cm} \text{ avec }  \tau'_1 \otimes \tau'_2 \in \JH \left( r_{(n-i,i), n} \left( \pi \right) \right), \text{et } \supp(\tau'_2) \subset \supp (\rho) \big\}.
\end{eqnarray*}
En effet, si $\tau'_1 \in \Irr (G_{n-i} )$ et $\tau'_2 \in \Irr (G_{i} )$, avec
$$\supp \left( \tau'_2 \right) \subset \supp \left( \rho \right),$$
d'apr\`es la proposition \ref{JA}, il existerait  $\tau_1 \in \Irr (G_{n-i} )$ et $\tau_2 \in \Irr (G_{i} )$, avec
$$\supp \left( \tau_2 \right)=\supp \left( \tau'_2 \right)  \subset \supp \left( \rho \right)$$
et $\Hom_{G_n}(\pi, \tau_1 \times \tau_2) \neq 0$.

Cette remarque nous permet de d\'efinir l'entier $l_{\pi}^{\supp(\rho)}$ pour $\pi$ une repr\'esentation de longueur finie (non n\'ecessairement irr\'eductible).
\end{remark}
\begin{proof}[D\'emonstration du du th\'eor\`eme \ref{l}]
Posons pour simplifier $l=l_{\pi}^{\supp(\rho)}$.
Soient $\tau_1 \in \Irr(G_{n-l}), \tau_2 \in \Irr(G_{l}),$ avec $\Hom_{G_n}(\pi, \tau_1 \times \tau_2) \neq 0$, et $\supp(\tau_2) \subset \supp (\rho) $. Alors,
\begin{enumerate}
\item La condition \eqref{C} nous dit que les segments de $\rho$ ne sont pas li\'es avec ceux de $\tau_2$ dans la param\'etrisation \`a la Langlands (resp. \`a la Zelevinsky). La repr\'esentation $\tau= \tau_2 \times \rho$ est donc irr\'eductible d'apr\`es \cite[Prop. 2.2]{Tadic} (resp. \cite[2.3.8]{Mi2}).
\item Par maximalit\'e de $l_{\pi}^{\supp(\rho)}$ et la remarque pr\'ec\'edente, pour tout $i\geq 1$ et tous $$\rho_1 \otimes \rho_2 \in \JH \left( r_{(n-l-i,i), n-l} \left( \tau_1 \right) \right),$$ $\rho_1 \in \Irr (G_{n-l-i} )$ et $\rho_2 \in \Irr (G_{i} )$, on a 
$$\supp \left( \rho_2 \right) \nsubseteq \supp \left( \tau \right).$$
\end{enumerate}

Le th\'eor\`eme d\'ecoule du fait que $\tau$ et $\tau_1$ sont alors deux repr\'esentations irr\'eductibles qui satisfont aux conditions du corollaire \ref{mmm1}. Leur induite n'a alors qu'un seul sous-module irr\'eductible et donc, $\pi \times \rho$, sous-module non nul de $\tau_1 \times \tau$, n'a, lui aussi, qu'un seul sous-module irr\'eductible. 
\end{proof}

%Posons $l=l_{\pi}^{\alpha}=\max \left\{ i : \exists \tau_1 \in \Irr(G_{n-i}), \tau_2 \in \Irr(G_{i}) \text{ avec } \Hom(\pi, \tau_1 \times \tau_2) \neq 0, \text{et } \supp(\tau_2) =\left\{ \alpha, \dots \alpha\right\} \right\}$ et soient $\tau_1 \in \Irr(G_{n-l}), \tau_2 \in \Irr(G_{l}) \text{ avec } \Hom(\pi, \tau_1 \times \tau_2) \neq 0, \text{et } \supp(\tau_2) =\left\{ \alpha, \dots \alpha\right\} $. Alors,
%\begin{enumerate}
%\item %La condition \ref{l}\eqref{C} nous dit que les segments de $\rho$ ne sont pas li\'es avec ceux de $\tau_2$.
%Par hypoth\`ese, la repr\'esentation $\tau= \tau_2 \times \rho$ est irr\'eductible %d'apr\`es \ref{despegados}.
%\item Par maximalit\'e de $l_{\pi}^{\alpha}$, pour tout $i\geq 1$ et tous $\rho_1 \otimes \rho_2 \in \JH \left( r_{(n-l-i,i), n-l} \left( \tau_1 \right) \right)$, $\rho_1 \in \Irr (G_{n-l-i} )$ et $\rho_2 \in \Irr (G_{i} )$, on a 
%$\supp \left( \rho_2 \right) \neq \left\{ \alpha, \dots \alpha\right\} $ (sinon $\overline{\Jac}_{\alpha}(\tau_1) \neq \emptyset$ et donc il existe $\tau_3 \in \Irr$ avec $\Hom(\tau_1, \tau_3 \times \alpha) \neq 0$, et donc $\Hom(\pi, \tau_3 \times \alpha \times \tau_2) \neq 0$).
%\end{enumerate}

%$\tau$ et $\tau_1$ sont donc deux repr\'esentations irr\'eductibles qui satisfont aux conditions de \ref{irred}. Le corollaire d\'ecoule du fait que $\pi \times \rho \subset \tau_1 \times \tau$ et que $\tau_1 \times \tau$ n'a qu'un seul sous-module irr\'eductible, d'apr\`es \ref{irred}.
%\end{proof}

\begin{remark}\label{mas1}
Si $V$ est l'unique sous-module irr\'eductible de $\pi \times \rho$, alors $l_{V}^{\supp(\rho)}=l_{\pi}^{\supp(\rho)}+t$, o\`u $t=\gr(\rho)$. En effet, l'in\'egalit\'e $l_{V}^{\supp(\rho)} \geq l_{\pi}^{\supp(\rho)}+t$ est claire. Pour l'autre, on remarque que, puisque le foncteur de Jacquet est exact, pour tout $\rho' \in \Irr$, si $\overline{\Jac}_{\rho'}(V) \neq 0$, alors $\overline{\Jac}_{\rho'}(\pi \times \rho) \neq 0$ et donc $l_{V}^{\supp(\rho)} \leq l_{\pi \times \rho}^{\supp(\rho)}$. Ce dernier entier, par le lemme g\'eom\'etrique, vaut $l_{\pi}^{\supp(\rho)}+t$.
\end{remark}

En passant \`a la contragr\'ediente on trouve:
\begin{corollary}\label{cociente} 
Avec les m\^emes hypoth\`eses, $\pi \times \rho$ a un unique quotient irr\'eductible et il appara\^it avec multiplicit\'e $1$.
\end{corollary}

De la m\^eme fa\c{c}on, en utilisant le foncteur de Jacquet oppos\'e et en red\'efinissant de fa\c{c}on appropri\'ee l'entier $l$, on montre le th\'eor\`eme ci-dessous.
\begin{theorem}\label{ttt}
Avec les hypoth\`eses de \ref{l}, $\rho \times \pi$ a un seul sous-module (resp. quotient) irr\'eductible et il appara\^it avec multiplicit\'e $1$.
\end{theorem}
On consid\`ere pourtant que les conditions ne sont pas n\'ecessaires et on se pose la question suivante: 
\begin{question} %\label{conj}
Est-ce que l'induite parabolique du produit tensoriel de deux repr\'esentations irr\'eductibles a toujours un seul sous-module irr\'eductible?
\end{question}
Remarquons que l'induite du produit tensoriel de \textit{trois} repr\'esentations irr\'eductibles peut avoir deux sous-modules irr\'eductibles (Par exemple, \textit{cf.} \cite[11.2]{Z1}, la repr\'esentation $1 \times |\,\, | \times 1$ de ${\rm GL}_3(F)$).

\section{Calcul explicite}\label{3}
Dans cette section on se propose de calculer les param\`etres de Langlands de l'unique sous-module irr\'eductible de $\pi \times \rho $. Le premier lemme est une transcription du lemme II.9. de \cite{MW} dans nos notations.

Nous introduisons quelques autres notations dont nous aurons besoin. On fixe une repr\'esentation cuspidale $\alpha$ de ${\rm GL}_n(D)$. D'apr\`es la remarque \ref{rectas}, il suffit de calculer les param\`etres de Langlands de l'unique sous-module de $\pi \times \alpha $ pour $\pi$ une repr\'esentation irr\'eductible dans  $\mathcal{R}(\alpha)$. Ainsi on peut identifier un segment $\Delta=\left\{\nu_{\alpha}^t \alpha, \nu_{\alpha}^{t+1}\alpha,\dots,\nu^{r}_{\alpha}\alpha\right\}$ \`a une suite $\left\{ t, t+1, \dots, r\right\}$. Dor\'enavant, s'il n'y pas d'ambigu\"it\'e, quand on a fix\'e une repr\'esentation cuspidale, on utilisera indiff\'eremment les deux notations. On notera aussi $b\left( \Delta \right) =t$ (ou $b\left( \Delta \right) =\nu_{\alpha}^t \alpha$), $e\left( \Delta \right) =r$ (ou $e\left( \Delta \right) =\nu_{\alpha}^r \alpha$) et $^+\Delta$ (resp. $^-\Delta$) le segment $^+\Delta=\left\{\nu_{\alpha}^{t-1} \alpha, \dots,\nu^{r}_{\alpha}\alpha\right\}$ (resp
 . $^-\Delta=\left\{\nu_{\alpha}^{t+1} \alpha, \dots,\nu^{r}_{\alpha}\alpha\right\}$).

On notera aussi $\left< \emptyset \right>^t $ la repr\'esentation triviale de $G_0$ et pour toute repr\'esentation $V$, $\widehat{V}=\left< \emptyset \right>^t $ (ce qui signifie, en pratique, qu'on \^ote la pr\'esence de $V$ dans la notation).
Le lemme ci-dessous est une adaptation du lemme II.9. de \cite{MW} dans nos notations, la preuve \'etant la m\^eme: 
 \begin{lemma}\label{pocos}
Soient c un entier, $\rho =\nu ^{c}\alpha $ la repr\'{e}sentation associ\'{e}e de ${\rm GL}\left( n,D\right) ,$ et $\pi $ une repr\'{e}sentation irr\'{e}ductible de ${\rm GL}\left( N-n,D\right) $ param\'{e}tr\'{e}e, \`a la Langlands,
par le multisegment rang\'{e} $m=\left\{ \Delta_{1},\ldots,
\Delta_{r}\right\} $. Alors l'ensemble des repr\'{e}sentations irr\'{e}ductibles de ${\rm GL}\left( N,D\right) $
qui sont isomorphes \`a des sous-repr\'{e}sentations de $ 
\pi \times \rho$ (resp. \`a des quotients de $\rho \times \pi$) est inclus dans l'ensemble des repr\'{e}sentations irr\'{e}ductibles suivantes:

$\left\langle \left\{ c\right\} ,\Delta _{1},\ldots ,\Delta
_{r}\right\rangle^t $

$\left\langle \Delta _{1},\ldots ,^{+}\Delta _{s},\ldots ,\Delta
_{r}\right\rangle^t $ o\`u $\Delta _{s}$ est un segment de $m$ d\'{e}butant \`a $%
c+1,$
et o\`u $s$ parcourt les entiers v\'erifiant la propri\'{e}t\'{e} suivante: notons $\Delta _{j\left( 1\right)
},\ldots ,\Delta _{j\left( t_\pi \right) }$ les segments de $m$ (dans l'ordre d\'{e}croissant) commen\c{c}ant par $c$; pour tout entier $v$ compris
entre $1$ et $r,$ on d\'{e}finit, inductivement sur $v,$ l'entier $i\left( v\right) $ comme \'{e}tant soit le plus grand entier
diff\'{e}rent de $i\left( 1\right) ,\ldots ,i\left( v-1\right) $, tel que $%
\Delta _{i\left( v\right) }$ commence par $c+1$ et soit pr\'{e}c\'{e}d\'{e}
par $\Delta _{j\left( v\right) }$, soit $i(v)=r+1$ si un tel entier n'existe pas; alors $s$ ne doit pas \^{e}tre l'un des entiers $i(v)$ qui viennent d'\^{e}tre d\'{e}finis.
\end{lemma}

On note $k(1),\dots,k(w_{\pi})$, ceux des entiers $j(v)$, pour $v\in \left\{ 1, \dots, t_{\pi} \right\}$, pour lesquels $i(v)$ est inf\'erieur ou \'egal \`a $r$, $k(1),\dots,k(w_{\pi})$ \'etant \'ecrits dans l'ordre croissant. Remarquons qu'on a $w_\pi \leq t_{\pi}$. On note aussi $h(1), \dots , h(w_{\pi})$ les $i(v)$ correspondants; $l(1),\dots,l(u_{\pi})$ les $i(v)$ diff\'erents de ceux qui viennent d'\^etre d\'efinis. On pose $l'_{\pi}=\left( t_{\pi} - w_{\pi}\right)$ et $s(1),\dots,s(l'_{\pi})$ les $i(v)$ diff\'erents de ceux qui viennent d'\^etre d\'efinis.

Remarquons que les entiers $t_{\pi}, w_\pi, u_\pi, l'_\pi$ ne d\'ependent que des segments $\left\{ \Delta_{1},\ldots ,\Delta_{r}\right\} $. On garde pourtant les sous-indices $\pi$ pour simplifier la notation. 

Le lemme suivant est la cl\'e de tout ce qui suit:
\begin{lemma}\label{lemme2} 
Soient $\Delta = \left\{ b,\dots,e\right\}, \Delta'=\left\{ b',\dots,e'\right\}$ deux segments tels que $\Delta$ pr\'ec\`ede $\Delta'$. Alors $\overline{\Jac}_b\left( \left< \Delta, \Delta' \right>^t\right) \neq \emptyset \Leftrightarrow b' \neq b+1$
\end{lemma}
\begin{proof}
On a une suite exacte
$$ 0 \rightarrow \left<\Delta \cup \Delta'\right>^t \times \left<\Delta \cap \Delta'\right>^t \rightarrow \left<\Delta\right>^t \times \left<\Delta' \right>^t \rightarrow \left< \Delta, \Delta' \right>^t \rightarrow 0.$$
Par le lemme g\'eom\'etrique on a : 
\begin{enumerate}
\item La repr\'esentation $r_{((e'-b)n,n), (e'-b+1)n}\left( \left< \Delta \right>^t \times \left< \Delta' \right>^t \right)$ est compos\'ee des repr\'esentations $\left( \left< ^-\Delta \right>^t \times \left< \Delta' \right>^t \right) \otimes b$ et $ \left( \left< \Delta \right>^t \times \left< ^-\Delta' \right>^t \right) \otimes b' $.
\item La repr\'esentation $r_{((e'-b)n,n), (e'-b+1)n} \left( \left<\Delta \cup \Delta'\right>^t \times \left<\Delta \cap \Delta'\right>^t \right) $ est compos\'ee des repr\'esentations $\left( \left<^-(\Delta \cup \Delta')\right>^t \times \left< \Delta \cap \Delta'\right>^t \right) \otimes b$ et $ \left( \left<(\Delta \cup \Delta')\right>^t \times \left<^-( \Delta \cap \Delta' )\right>^t \right) \otimes b'$ (si $\Delta \cap \Delta' \neq \emptyset$).
\end{enumerate}
et donc 
\begin{eqnarray*}
\overline{\Jac}_b\left( \left< \Delta \right>^t \times \left< \Delta' \right>^t \right)& =& \left( \left< ^-\Delta \right>^t \times \left< \Delta' \right>^t \right) \otimes b \\
\overline{\Jac}_b\left( \left<\Delta \cup \Delta'\right>^t \times \left<\Delta \cap \Delta'\right>^t \right) &=& \left( \left<^-(\Delta \cup \Delta')\right>^t \times \left< \Delta \cap \Delta'\right>^t \right) \otimes b .
\end{eqnarray*}

Ainsi, par exactitude du foncteur de Jacquet, on a $\overline{\Jac}_b \left( \left< \Delta, \Delta' \right>^t\right) = \left< ^-\Delta , \Delta' \right>^t  \otimes b $ si, et seulement si, $\left<^-\Delta\right>^t \times \left<\Delta'\right>^t \neq \left<^-(\Delta \cup \Delta')\right>^t \times \left< \Delta \cap \Delta'\right>^t $, \textit{i.e,} si $b' \neq b+1$.
\end{proof}

Cela implique le
\begin{corollary}\label{ineq}
Avec les notations de \ref{l} et \ref{pocos}, $\rho$ \'etant toujours une repr\'esentation cuspidale de $G_n$, on a $l^{\left\{\rho\right\}}_{\pi} \leq n l'_{\pi}$.
\end{corollary}

\begin{proof}
Pour tout entier $i$ compris entre $1$ et $r$, on pose

$$V_{i}=\left\{ 
\begin{array}{l}
\Delta _{i}\text{, si }i\notin \left\{ h\left( v\right) :1\leq v\leq
w_{\pi}\right\} \cup \left\{ k\left( v\right) :1\leq v\leq w_{\pi}\right\} \\ 
\left\langle \Delta _{h\left( v\right) },\Delta _{k\left( v\right)
}\right\rangle^t ,\text{ si }i=h\left( v\right) ,1\leq v\leq w_{\pi}, \\ 
\left\langle \emptyset \right\rangle^t ,\text{si }i=k\left( v\right) ,1\leq
v\leq w_{\pi}
\end{array}
\right. $$
D'apr\`es \cite[Preuve du Lemme II.9]{MW} $\pi$ est un quotient de 
$V_1 \times \dots \times V_r $, et donc, par exactitude du foncteur de Jacquet $l^{\left\{\rho\right\}}_{\pi} \leq l^{\left\{\rho\right\}}_{V_1 \times \dots \times V_r}$.

D'un autre c\^ot\'e, par le lemme \ref{lemme2}, les seuls $V_i$ tels que $\overline{\Jac}_{\rho}(V_i) \neq \emptyset$ sont les $V_{s(i)}$ avec $1 \leq i \leq l'_{\pi}$ et donc, par le lemme g\'eom\'etrique
$$l^{\left\{\rho\right\}}_{V_1 \times \dots \times V_r}=n l'_{\pi},$$
qui prouve le lemme.
\end{proof}

\qquad

On d\'efinit deux op\'erateurs:
\begin{eqnarray*}
Q_c&:& M(S) \longrightarrow M(S)\\
S_c&:& M(S) \longrightarrow M(S),
\end{eqnarray*}
par les formules ci-dessous. Notons $\pi=  \left<\Delta_1, \dots, \Delta_r\right>^t$; on rappelle qu'on a d\'efini en \ref{pocos} des entiers $u_\pi, l'_\pi$. Alors:
\begin{eqnarray*}
Q_c \left(\left\{ \Delta_1, \dots, \Delta_r\right\}\right)&=& \begin{cases}
\left\{ \Delta_1, \dots, ^+\Delta_{l(1)}, \dots, \Delta_r\right\} &\text{si } u_{\pi}\geq 1\\
\left\{ \left\{c\right\}, \Delta_1, \dots, \Delta_r\right\} &\text{si } u_{\pi}=0
\end{cases}\\
S_c \left(\left\{ \Delta_1, \dots, \Delta_r\right\}\right)&=& \begin{cases}
\left\{ \Delta_1, \dots,^- \Delta_{s(l'_{\pi})}, \dots, \Delta_r\right\} &\text{si } l'_{\pi}\geq 1\\
\emptyset &\text{si } l'_{\pi}=0,
\end{cases}
\end{eqnarray*}
o\`u le multisegment $\left\{ \Delta_1, \dots, \Delta_r\right\}$ est suppos\'e rang\'e.

Le lemme ci-dessous explicite, en termes des segments, l'action de $S_c$ sur un multisegment:

\begin{lemma}\label{ope}
Soit $\left\{ \Delta_1, \dots, \Delta_{r}\right\}$ un multisegment. Notons $\left\{ \Delta'_1, \dots, \Delta'_{r'}\right\}$ le multisegment $S_c(\left\{ \Delta_1, \dots, \Delta_r\right\})$, et $\pi=\left<\Delta_1, \dots, \Delta_r\right>^t$ Alors,
\begin{enumerate}
\item $\Delta'_{h(v)}=\Delta_{h(v)}$ et $\Delta'_{k(v)}=\Delta_{k(v)}$ pour tout $v=1, \dots, t_{\pi}$,
\item $\Delta'_{l(v)}=\begin{cases} ^-\Delta'_{s(l'_{\pi})}& \text{si } v=1\\
 \Delta'_{l(v-1)}& \text{si } v=2, \dots, w_{\pi}+1.
\end{cases}$
\end{enumerate}
\end{lemma}
\begin{proof}
On a $\Delta'_{j(v)}=\begin{cases} \Delta'_{j(v)}& \text{si } v < j^{-1}\circ s(l'_{\pi})\\
\Delta'_{j(v+1)}& \text{si } v \geq j^{-1}\circ s(l'_{\pi}).\end{cases}$

Montrons (1) par r\'ecurrence sur $v$:

Par construction de $k$, il n'y a pas de segment $\Delta_{\alpha}$ commen\c{c}ant par $c+1$ qui pr\'ec\`ede $\Delta_{j(t)}$ pour $j(t)<k(1)$. Par minimalit\'e dans l'ensemble des $\Delta_{s(\alpha)}$, $^-\Delta_{s(l'_{\pi})}$ ne pr\'ec\`ede aucun des $\Delta_{j(t)}$ pour $j(t) < k(1)$. D'o\`u
$$\Delta'_{k(1)}=\Delta_{k(1)}.$$
$\Delta_{h(1)}$ est un segment dans $\left\{\Delta'_{1}, \dots , \Delta'_{r}\right\}$ commen\c{c}ant par $c+1$, pr\'ec\'edant $\Delta'_{k(1)}$ et minimal parmi les $\Delta_{\alpha}$ v\'erifiant ces deux propri\'et\'es.
\begin{itemize}
\item Si $k(1)< s(l'_{\pi})$, alors $^-\Delta_{s(l'_{\pi})}$ ne pr\'ec\`ede pas $\Delta'_{k(1)}$.
\item Si $k(1)> s(l'_{\pi})$, puisque $\Delta_{h(1)}$ ne pr\'ec\`ede pas $\Delta_{s(l'_{\pi})}$ (par construction de $h(1)$), alors $\Delta_{h(1)} \leq ^-\Delta_{s(l'_{\pi})},$
\end{itemize}
d'o\`u $$\Delta'_{h(1)}=\Delta_{h(1)}.$$

Supposons maintenant que $\Delta'_{h(i)}=\Delta_{h(i)}$ et $\Delta'_{k(i)}=\Delta_{k(i)}$ pour tout $i\in \{1, \dots, v\}$, et montrons que $\Delta'_{h(v+1)}=\Delta_{h(v+1)}$ et $\Delta'_{k(v+1)}=\Delta_{k(v+1)}$. 

Il n'y a pas de segment $\Delta_{\alpha}$ commen\c{c}ant par $c+1$ qui pr\'ec\`ede $\Delta_{j(t)}$ pour $j(t)<k(v+1)$ et $j(t) \notin \left\{k(1), \dots, k(v) \right\}$. Par minimalit\'e dans l'ensemble des $\Delta_{s(\alpha)}$, le segment $^-\Delta_{s(l'_\pi)}$ ne pr\'ec\`ede aucun des $\Delta_{j(t)}$ pour $j(t) < k(v+1)$ et $j(t) \notin \left\{k(1), \dots, k(v) \right\}$ (ce sont des "$\Delta_{s(\beta)}$"\dots). D'o\`u
$$\Delta'_{k(v+1)}=\Delta_{k(v+1)}.$$
$\Delta_{h(v+1)}$ est un segment dans $\left\{\Delta'_{1}, \dots , \Delta'_{r}\right\}$ commen\c{c}ant par $c+1$, pr\'ec\'edant $\Delta'_{k(v+1)}$ diff\'erent des $\Delta_{h(\beta)}$, pour $\beta \leq v$ et minimal parmi les $\Delta_{\alpha}$ v\'erifiant ces trois propri\'et\'es.
\begin{itemize}
\item Si $k(v+1)< s(l'_{\pi})$, alors $^-\Delta_{s(l'_{\pi})}$ ne pr\'ec\`ede pas $\Delta'_{k(v+1)}$.
\item Si $k(v+1)> s(l'_{\pi})$, puisque $\Delta_{h(v+1)}$ ne pr\'ec\`ede pas $\Delta_{s(l'_{\pi})}$ (par construction de $h(v+1)$), alors $\Delta_{h(v+1)} \leq ^-\Delta_{s(l'_{\pi})},$
\end{itemize}
d'o\`u $$\Delta'_{h(v+1)}=\Delta_{h(v+1)}.$$

Pour montrer (2) il suffit, d'apr\`es ce qui pr\'ec\`ede, de remarquer que $\Delta_{l(1)} \leq ^-\Delta_{s(l'_{\pi})}.$
\end{proof}

Le corollaire suivant est une cons\'equence imm\'ediate:
\begin{corollary}\label{ayuda}
\begin{enumerate}
\item Si $l'_{\pi}\geq 1$, alors $Q_c \circ S_c = Id_{M(S)}$.
\item Si $l'_{\pi}\geq 1$, alors $l'_{\left<S_c \left(\left\{ \Delta_1, \dots, \Delta_r\right\}\right)\right>^t}=l'_{\left<\Delta_1, \dots, \Delta_r\right>^t }-1$.
\item $l'_{\left<Q_c \left(\left\{ \Delta_1, \dots, \Delta_r\right\}\right)\right>^t}=l'_{\left<\Delta_1, \dots, \Delta_r\right>^t }+1$.
\item Si $\omega_i $ est l'un des sous-modules d\'ecrits dans le lemme \ref{pocos} et $\omega_i \ncong \left<Q_c \left(\left\{ \Delta_1, \dots, \Delta_r\right\}\right)\right>^t$, alors $l'_{\omega_i}=l'_{\left<\Delta_1, \dots, \Delta_r\right>^t }$.
\end{enumerate}
\end{corollary}

Le th\'eor\`eme suivant est la conclusion de notre \'etude soigneuse:
\begin{theorem}\label{unicite} Soient $\rho$ une repr\'esentation cuspidale de $G_n$, $\pi=  \left<\Delta_1, \dots, \Delta_r\right>^t$. Avec les notations pr\'ec\'edentes, on a :
\begin{enumerate}
\item $l^{\left\{\rho\right\}}_{\pi} = n l'_{\pi}$
\item L'unique sous-module irr\'eductible de $\left<\Delta_1, \dots, \Delta_r\right>^t \times \rho$ est \newline
$\left<Q_c \left(\left\{ \Delta_1, \dots, \Delta_r\right\}\right)\right>^t.$
\item Si $l'_{\pi}\geq 1$, $\left<\Delta_1, \dots, \Delta_r\right>^t$ est un sous-module irr\'eductible de $$\left<S_c \left(\left\{ \Delta_1, \dots, \Delta_r\right\}\right)\right> \times \rho.$$
\item L'unique quotient irr\'eductible de $ \rho \times \left<\Delta_1, \dots, \Delta_r\right>^t $ est
$$\left<Q_c \left(\left\{ \Delta_1, \dots, \Delta_r\right\}\right)\right>^t.$$
\end{enumerate}
\end{theorem}
\begin{proof}
(3) est une cons\'equence de (2) et \ref{ayuda}.(1). 

Montrons (1) et (2) par r\'ecurrence sur $l'_{\pi}$.

Si $l'_{\pi}=0$, d'apr\`es \ref{ineq} on a $l^{\left\{\rho\right\}}_{\pi} =0$. Si $\pi \times \rho$ avait pour sous-modules l'un des $\omega_i \neq \left<Q_c \left(\left\{ \Delta_1, \dots, \Delta_r\right\}\right)\right>^t$, on aurait, par \ref{ayuda}.(4), l'\'galit\'e $l'_{\omega_i}=0$, donc $l^{\left\{\rho\right\}}_{\omega_i} =0$ ce qui est absurde par \ref{mas1}.

Supposons $l'_{\pi}=k >0$. Alors, par \ref{ayuda}.(2), $$l'_{\left<S_c \left(\left\{ \Delta_1, \dots, \Delta_r\right\}\right)\right>^t}=l'_{\left<\Delta_1, \dots, \Delta_r\right>^t }-1$$ donc, par hypoth\`ese de r\'ecurrence, le th\'eor\`eme est vrai pour la repr\'esentation $\left<S_c \left(\left\{ \Delta_1, \dots, \Delta_r\right\}\right)\right>^t$, \textit{i.e.} 
\begin{enumerate}
\item $\left<Q_c \circ S_c \left(\left\{ \Delta_1, \dots, \Delta_r\right\}\right)\right>^t $ est l'unique sous-module de $\pi \times \rho$. Or, $$\left<Q_c \circ S_c \left(\left\{ \Delta_1, \dots, \Delta_r\right\}\right)\right>^t=\pi $$ par \ref{ayuda}.(1).
\item $l^{\left\{\rho\right\}}_{\left<S_c \left(\left\{ \Delta_1, \dots, \Delta_r\right\}\right)\right>^t}=n(k-1).$
\end{enumerate}
Par \ref{mas1}, on a que $l^{\left\{\rho\right\}}_{\pi}=n(k-1)+n=kn$.

Finalement supposons que $\pi \times \rho$ a pour sous-module l'un des $\omega_i \neq \left<Q_c \left(\left\{ \Delta_1, \dots, \Delta_r\right\}\right)\right>^t$. Alors, par \ref{ayuda}.(4) on aurait $l'_{\omega_i}=l'_{\left<\Delta_1, \dots, \Delta_r\right>^t }=k$. On vient de montrer que, si $l'_{\omega_i}=k$, on trouve par hypoth\`ese de r\'ecurrence $l^{\left\{\rho\right\}}_{\omega_i} =kn=l^{\left\{\rho\right\}}_{\pi}$ ce qui est absurde, par \ref{mas1}.

(4) est une cons\'equence de (1), de \ref{ayuda}.(4) et de \ref{ttt} ce qui ach\`eve la d\'emonstration du th\'eor\`eme.
\end{proof}
\section{Applications}\label{aap}
La proposition suivante est due \`a \cite{GK} dans le cas o\`u $D=F$, mais leur preuve n'est pas valable quand $D$ n'est pas commutatif car la transpos\'ee d'une matrice inversible n'est plus forc\'ement inversible. Elle d\'ecoule imm\'ediatement des parties 2. et 4. du th\'eor\`eme pr\'ec\'edent.
\begin{proposition}\label{cambio}
Soient $\pi, \pi' \in \Irr, \rho \in \mathcal{C}$. Les conditions suivantes sont \'equivalentes:\begin{enumerate}
\item $\Hom \left( \pi' , \pi \times \rho \right) \neq 0;$
\item $\Hom \left( \rho \times \pi, \pi' \right) \neq 0. $
\end{enumerate}
\end{proposition}
\begin{remark}
On pense, bien s\^ur, que l'hypoth\`ese $\rho$ cuspidale n'est pas n\'ecessaire.
\end{remark}

Notons $\Delta_{j'(1)},\dots,\Delta_{j'(t'_{\pi})}$ les segments de $m$ (dans l'ordre d\'ecroissant) se terminant par $c$; pour tout entier $v$ compris entre $1$ et $t'_{\pi}$, on d\'efinit inductivement sur $v$, l'entier $i'(v)$ comme \'etant soit le plus petit entier diff\'erent de $i'(1),\dots,i'(v-1)$, tel que $\Delta_{i'(v)}$ termine par $c-1$ et pr\'ec\`ede $\Delta_{j'(t'_{\pi}-v+1)}$, soit $i'(v)=r+1$ si un tel entier n'existe pas. On note $l'(1),\dots,l'(u'_{\pi})$ les $i'(v)$ diff\'erents de ceux qui viennent d'\^etre d\'efinis. 

$\omega'_0= \left< \left\{ c \right\},\Delta_1, \dots, \Delta_r \right>^t$.

$\omega'_i= \left< \Delta_1, \dots, \Delta_{l'(i)}^+,\dots, \Delta_r \right>^t$, o\`u $i \in 1,\dots, u'_{\pi}$.

On a d\'efini les entiers $j'(v),i'(v)$ pour que, en appliquant la contragr\'ediente \`a partir du th\'eor\`eme \ref{unicite}(2), on trouve:
\begin{corollary}
L'unique quotient irr\'eductible de $\pi \times \rho$ est $\omega'_1$, si $u'_{\pi}>0$ o\`u $\omega'_0$, si $u'_{\pi}=0$.
\end{corollary}

Si $\pi \times \rho$ est irr\'eductible, alors il est clair que $u'_{\pi}=u_{\pi}=0$. R\'eciproquement,
si $u'_{\pi}=u_{\pi}=0$ on a que $\omega_0=\omega'_0$ est le seul quotient et sous-module irr\'eductible de $\pi \times \rho$. Or, d'apr\`es \ref{l}, $\omega_0$ appara\^it avec multiplicit\'e $1$ dans $\JH(\pi \times \rho)$, d'o\`u
\begin{theorem}\label{irredu}
$\pi \times \rho$ est irr\'eductible si, et seulement si, $u'_{\pi}=u_{\pi}=0$.
\end{theorem}

Maintenant il est clair que tout ce qui pr\'ec\`ede dans cette section est aussi vrai pour les param\'etrisations \textit{\`a la Zelevinsky}. Il suffit de remplacer, dans les \'enonc\'es et les preuves,
\begin{enumerate}
\item Le mot \textit{quotient} par \textit{sous-repr\'esentation},
\item le mot \textit{sous-repr\'esentation} par \textit{quotient},
\item le sens de toutes les fl\`eches
\item le symbole $\left< \, \, \right>^t$ par $\left< \, \, \right>$,
\item le foncteur $r$ (resp. $\overline{r}$) par le foncteur $\overline{r}$ (resp. $r$).
\end{enumerate}
Ainsi, on montre un th\'eor\`eme similaire au th\'eor\`eme \ref{unicite}:
\begin{theorem}\label{ueu}
\begin{enumerate}
\item L'unique quotient irr\'eductible de la repr\'esentation $\left<\Delta_1, \dots, \Delta_r\right> \times \rho$ est 
$\left<Q_c \left(\left\{ \Delta_1, \dots, \Delta_r\right\}\right)\right>.$
\item L'unique sous-module irr\'eductible de $ \rho \times \left<\Delta_1, \dots, \Delta_r\right> $ est \newline
$\left<Q_c \left(\left\{ \Delta_1, \dots, \Delta_r\right\}\right)\right>.$
\end{enumerate}
\end{theorem}
\begin{corollary}\label{co}
Soient $\rho$ une repr\'esentation cuspidale et $\pi$ une repr\'esentation irr\'eductible. Notons $\tau$ l'involution de Zelevinsky (voir \ref{zze}). Les conditions suivantes sont \'equivalentes:
\begin{enumerate}
\item $V$ est un sous-module irr\'eductible de $\rho \times \pi$.
\item $\tau \left(V\right)$ est un sous-module irr\'eductible de $\tau \left( \pi\right) \times \tau \left(\rho\right).$
\end{enumerate}
\end{corollary}
\begin{proof}
C'est une cons\'equence de \ref{ueu}.(2), \ref{unicite}.(2) et \ref{rectas}.
\end{proof}

La cons\'equence de ce corollaire est que les r\'esultats du papier \cite{MW} sont valables pour des repr\'esentations complexes de ${\rm GL}_r(D)$ comme c'\'etait conjectur\'e dans \cite[Conjecture 3.6]{Tadic}. En effet, toute la partie I de \cite{MW} \'etait consacr\'e \`a la preuve du corollaire \ref{co} pour $\pi$ et $\rho$ des repr\'esentations \textit{irr\'eductibles} d'une certaine alg\`ebre de Hecke mais pour la preuve du th\'eor\`eme qui suit, ils n'utilisaient que le corollaire pr\'ec\'edent.
\begin{theorem}\label{geometrique}
L'involution $\tau$ v\'erifie la description g\'eom\'etrique de \cite{Z2} et la description combinatoire de \cite{MW}. 
\end{theorem}
\begin{proof}
Il suffit de changer dans la preuve du th\'eor\`eme \cite[II.13]{MW} la proposition \cite[I.7.3]{MW}  par le corollaire pr\'ec\'edent.
\end{proof}
\appendix
\section{La correspondance th\^eta}
Ici on montre un lemme qui nous aidera \`a calculer explicitement la correspondance th\^eta dans \cite{Mi1}. On continue avec les notations de la section \ref{3}. Soient $a,b,c$ des entiers.
\begin{lemma}\label{com}
Soit $\left\{\Delta'_1, \dots, \Delta'_{r'}\right\}$ un multisegment et posons $\pi$ la sous-repr\'esen\-tation irr\'eductible de $\left<\Delta'_1, \dots, \Delta'_{r'}\right>^t\times \left<c\right>$ et $\pi'$ la sous-repr\'esen\-tation irr\'eductible de $\left<b, \dots, b-a, \Delta'_1, \dots, \Delta'_{r'}\right>^t\times \left<c\right>$.

Si $c \notin \left\{b, b-a-1\right\}$ et $\pi=\left<\Delta_1, \dots, \Delta_r\right>^t$, alors
$$\pi'=\left<b, \dots, b-a,\Delta_1, \dots, \Delta_r\right>^t.$$
\end{lemma}
\begin{proof}
La condition $c \notin \left\{b, b-a-1\right\}$ \'equivaut au fait que $c$ et $c+1$ appartiennent tous les deux \`a $\left\{b, b-1, \dots, b-a, b-a-1\right\}$ ou aucun des deux. 
Ainsi, si l'on note
$\left\{ \delta_1, \dots, \delta_n\right\}$ (resp. $\left\{ \delta'_1, \dots, \delta'_{n'} \right\}$) le sous-ensemble de $\left\{ \Delta^{\prime}_1, \dots, \Delta^{\prime}_{r'}\right\}$ (resp. $\left\{b, \dots, b-a, \Delta^{\prime}_1, \dots, \Delta^{\prime}_{r'}\right\}$ des segments commen\c{c}ant par $c$ ou $c+1$ on a que:
\begin{enumerate}
\item Si $c \in \left\{b, b-1, \dots, b-a, b-a-1\right\} $, alors 
$$\left\{c, c+1, \delta_1, \dots, \delta_n\right\}=\left\{ \delta'_1, \dots, \delta'_{n'} \right\}.$$
\item Si $c \notin \left\{b, b-1, \dots, b-a, b-a-1\right\} $, alors 
$$\left\{\delta_1, \dots, \delta_n\right\}=\left\{ \delta'_1, \dots, \delta'_{n'} \right\}.$$
\end{enumerate}
Dans le deuxi\`eme cas le lemme est une cons\'equence imm\'ediate de \ref{unicite}. Dans le premier cas, il est \'evident, par r\'ecurrence comme dans \ref{ope}, que
$\delta^{\prime}_{l(v)}=\delta_{l(v)}$ pour $v=1, \dots, \omega_{\left<\delta'_1, \dots, \delta'_{n'}\right>^t}$ et $\omega_{\left<\delta_1, \dots, \delta_n\right>^t}= \omega_{\left<\delta'_1, \dots, \delta'_{n'}\right>^t}$.
On ach\`eve la d\'emonstration avec le th\'eor\`eme \ref{unicite}.
\end{proof}
On r\'ecrit le lemme dans les notations qu'on utilisera dans \cite{Mi1}. On rappelle qu'on a d\'efini, pour $g \in D^\times$, $ \nu(g)=\left| \ry_{D}(g)\right|_F$.
%\begin{corollary}
%Soit $\pi$ l'unique sous-repr\'esentation irr\'eductible de $ \left<\nu^{-c}\right>\times \left<\Delta'_1, \dots, \Delta'_{r'}\right>^t$ et $\pi'$ la sous-repr\'esentation irr\'eductible de $$\left<\nu^b, \dots, \nu^{b-a}, \widetilde{\Delta'_1}, \dots, \widetilde{\Delta'_{r'}}\right>^t\times \left<\nu^{c}\right>,$$ $a \geq 0$.
%Si $c \notin \left\{b, b-a-1\right\}$ et $\pi=\left<\Delta_1, \dots, \Delta_r\right>^t$, alors
%$$\pi'=\left<\nu^b, \dots, \nu^{b-a},\widetilde{\Delta_1}, \dots, \widetilde{\Delta_r}\right>^t.$$
%\end{corollary}
%\begin{proof}
%La sous-repr\'esentation irr\'eductible de $ \left<\nu^{-c}\right>\times \left<\Delta'_1, \dots, \Delta'_{r'}\right>^t$ est, par passage \`a la contragr\'ediente et \ref{cambio}, l'unique sous-repr\'esentation irr\'eductible de $\left<\widetilde{\Delta'_1}, \dots, \widetilde{\Delta'_{r'}}\right>^t\times \left<\nu^{c}\right>$. Le corollaire est donc une cons\'equence imm\'ediate de \ref{com}.
%\end{proof}

\begin{defi}
Soit $\pi\in \Irr(G_n)$, quotient de Langlands de $\tau_1 \times \dots \times \tau_r$, o\`u $\tau_1,\dots ,\tau_r$ sont des repr\'esentations essentiellement de carr\'e int\'egrable. Notons alors $\theta^\ast _m(\pi)$ le quotient de Langlands de $$\nu ^{\frac{m-2n-1}{2}} \times \dots \times \nu ^{\frac{-m+1}{2}} \times\nu ^{\frac{m-n}{2}} \widetilde{\tau_1} \times \dots \times \nu ^{\frac{m-n}{2}} \widetilde{\tau_r}.$$
\end{defi}
\begin{corollary}\label{comb}
Soient $\rho$ une repr\'esentation cuspidale de $G_p$, $\rho \neq \begin{cases} \nu ^{\frac{n+1}{2}}\\ \nu ^{\frac{2m-n+1}{2}}\end{cases}$, $\pi_1\in \Irr(G_{n-p})$, et $\pi$ l'unique sous-repr\'esentation irr\'eductible de $\rho \times \pi_1$. Notons $\pi'$ l'unique sous-repr\'esentation irr\'eductible de $\nu^{\frac{-p}{2}}\theta_{m-p}^\ast(\nu^{\frac{-p}{2}}\pi_1) \times \nu^{\frac{m-n}{2}}\widetilde{\rho}$. Alors
$$\pi'= \theta_{m}^\ast(\pi).$$
\end{corollary}
\begin{proof}
En effet, par \ref{rectas}, on peut supposer que pour tout $i\leq r$, $\supp(\tau_i) \subset \mathcal{R}\left(\nu ^{\frac{n+1}{2}}\right)$. De m\^eme $ \rho \in \mathcal{R}\left(\nu ^{\frac{n+1}{2}}\right)$, sinon le r\'esultat est trivial. Soit $c \in \mathbb{R}$ tel que $\rho = \nu ^{-c}$ et soient $\Delta'_1, \dots, \Delta'_r$ des segments tels que $\pi_1= \left<\Delta'_1, \dots, \Delta'_{r'}\right>^t$. Alors $\pi$ est, par \ref{cambio}, l'unique sous-module irr\'eductible de $\left<\widetilde{\Delta'_1}, \dots, \widetilde{\Delta'_{r'}}\right>^t\times \left<\nu^{c}\right>$. Le corollaire d\'ecoule maintenant du lemme \ref{com}, avec $b=\frac{-n-1}{2}$ et $a=m-n$.

\end{proof}
Est-il possible de montrer ce th\'eor\`eme sans utiliser tous les calculs de la section pr\'ec\'edente?


\begin{thebibliography}{99}

\bibitem[Aub]{Aubert} A.-M. Aubert, \textit{Dualit\'e dans le groupe de Grothendieck de la cat\'egorie des repr\'esentations lisses de longueur finie d'un groupe r\'eductif $p$-adique}, Trans. Amer. Math. Soc. {\bf347} 
(1995), 2179-2189 et \textit{Erratum}, Trans. Amer. Math. Soc. {\bf348} 
(1995), 4687-4690.

\bibitem[Ber]{ber} I.N. Bernstein, \textit{Representations of $p$-adic groups}, Notes by K.E. Rumelhart, Harvard Univ. 1992.

\bibitem[Cas]{cas} W. Casselman, \textit{Introduction to the theory of
admissible representations of }$p$\textit{-adic reductive groups}, preprint.

\bibitem[GK]{GK} I.M. Gelfand, D.A. Kazhdan, \textit{Representations of the group ${\rm GL}(n,K)$ where $K$ is a local field,}, Lie Groups and their representations, I.M. Gelfand ed., London 1975.

\bibitem[Mi1]{Mi1} A. M\'inguez, \textit{Correspondance de Howe explicite: paires duales de type II},  pr\'epublication avril 2007.

\bibitem[Mi2]{Mi2} A. M\'inguez, \textit{Correspondance de Howe $l$-modulaire: paires duales de type II}, th\`ese, Orsay 2006.

\bibitem[Moe]{Moe} C. M\oe glin, \textit{Normalisation des op\'erateurs d'entrelacement et r\'eductibilit\'e des induites de cuspidales; le cas des groupes classiques ${\rm p}$-adiques}. Ann. Math. 151, 817--847 (2000)

\bibitem[MVW]{MVW} C. Moeglin, M.F. Vign\'{e}ras, J.L. Waldspurger, \textit{%
Correspondance de Howe sur un corps }$p$\textit{-adique}, LNM 1291,
Springer-Verlag, 1987.

\bibitem[MW]{MW} C. M\oe glin, J.-L. Waldspurger, \textit{Sur l'involution de Zelevinsky,} J. Reine Angew. Math. {\bf372} 
(1986), 136-177.

\bibitem[Se]{Se4} V. S\'echerre, \textit{Proof of the Tadic conjecture U0 on the unitary dual if ${\rm GL}_m(D)$}, pr\'epublication 2006.

\bibitem[Ta1]{Tadic} M. Tadi\'c \textit{Induced representations of ${\rm GL}(n,A)$ for $p$-adic division algebras} J. Reine Angew. Math. {\bf405} 
(1990), 48-77.

\bibitem[Ta2]{Tad2} M. Tadi\'c \textit{Representation theory of ${\rm GL}(n)$ over a $p$-adic division algebra and unitarity in the Jacquet-Langlands correspondence}. Pacific J. Math. 223 (2006), no. 1, 167-200. 

\bibitem[Ze1]{Z1} A.V. Zelevinsky, \textit{Induced Representations of Reductive}
$p$\textit{-Adic Groups II}, Ann. Scient. Ec. Norm. Sup., 4$^\text{{e}}$
serie, t. 13, 1980, 165-210.

\bibitem[Ze2]{Z2} A.V. Zelevinsky, \textit{$p$-adic analogue of the Kazhdan-Lusztig conjecture}, Funct. Anal. Appl. {\bf15}, (1981) 83-92.



\end{thebibliography}
\end{document}